\documentclass{amsart}
\usepackage{amsfonts}
\usepackage{amsmath,amsthm}
\usepackage{amsfonts,amssymb}
\usepackage{multirow}
\usepackage{booktabs}
\usepackage{enumerate}

\newtheorem{thm}{Theorem}[section]
\newtheorem{cor}[thm]{Corollary}
\newtheorem{lem}[thm]{Lemma}
\newtheorem{prop}[thm]{Proposition}

\newtheorem{defn}[thm]{Definition}
\newtheorem{rem}{Remark}

\theoremstyle{remark}

\numberwithin{equation}{section}

\newcommand{\al}{\alpha}

\def\sph{\mathbb{S}^{d}}

\def\vz{\varepsilon}

\def\lz{\lambda}
\def\Lz{\Lambda}
\def\Oz{\Omega}

\def\dz{\delta}
\def\az{\alpha}
\def\gz{\gamma}
\def\Gz{\Gamma}

\def\sz{\sigma}

\def\f{\frac}
\def\({\Bigl(}
\def \){ \Bigr)}

\def\Ga{\Gamma}

\def\sub{\substack}

 \def\RR{{\mathbb R}}

\def\sz{\sigma}

\def\ss{{\Bbb S}^{d}}

\def\x{{\bf x}}
\def\y{{\bf y}}

\begin{document}
\def\RR{\mathbb{R}}
\def\Exp{\text{Exp}}
\def\FF{\mathcal{F}_\al}

\title[] { Optimal  quadrature errors and sampling numbers for Sobolev spaces with logarithmic perturbation on spheres}

\author{Jiaxin Geng} \address{ School of Mathematical Sciences, Capital Normal
University, Beijing 100048,
 China.}
 \email{gengjiaxin1208@163.com}

 \author{Yun Ling} \address{ School of Mathematical Sciences, Capital Normal
University, Beijing 100048,
 China.}
 \email{18158616224@163.com}

\author{Jiansong Li} \address{ School of Mathematical Sciences, Capital Normal
University, Beijing 100048,
 China}
\email{2210501007@cnu.edu.cn}

\author{Heping Wang} \address{ School of Mathematical Sciences, Capital Normal
University,
Beijing 100048,
 China.}
\email{wanghp@cnu.edu.cn}

\keywords{Worst-case error,
Numerical integration,
Quadrature formulas, Reproducing kernel Hilbert spaces} \subjclass[2010]{41A63; 65C05;
65D15;  65Y20}

\begin{abstract}
In this paper, we study
 optimal quadrature errors, approximation numbers, and sampling numbers in $L_2(\ss)$ for  Sobolev spaces ${\rm
H}^{\az,\beta}(\ss)$  with logarithmic perturbation on the unit
sphere $\ss$ in $\Bbb R^{d+1}$.  First we obtain strong
equivalences of the approximation numbers for ${\rm
H}^{\az,\beta}(\ss)$ with $\az>0$, which gives a clue to Open
problem 3 as posed by Krieg and Vyb\'iral in \cite{KV}. Second,
for  the optimal quadrature errors for ${\rm H}^{\az,\beta}(\ss)$,
we use  the ``fooling" function technique to get lower bounds in
the case  $\az>d/2$, and apply Hilbert space structure and
Vyb\'iral's theorem about Schur product theory to obtain lower
bounds in the case $\az=d/2,\,\beta>1/2$ of small smoothness,
which confirms the conjecture as posed by Grabner and Stepanyukin
in \cite{GS} and solves Open problem 2 in \cite{KV}. Finally, we
employ
 the weighted
least squares operators and the least squares quadrature rules to
obtain approximation theorems and quadrature errors for ${\rm
H}^{\az,\beta}(\ss)$ with $\az>d/2$ or $\az=d/2,\,\beta>1/2$,
which are order optimal.

.

\end{abstract}

\maketitle
\input amssym.def

\section{Introduction and main results}

\

For $d\ge 2$, denote $\Bbb S^d=\{\x\in \Bbb R^{d+1}
|\,\|\x\|=(x_1^2+\cdots +x_{d+1}^2)^{1/2}=1\}$ by the unit sphere.
 The integral of a continuous function $f:{\Bbb
S}^{d}\to \Bbb R $ is denoted by
$${\rm INT}(f):=\int_{{\Bbb S}^{d}}f({\bf{x}})\,  d\sigma_{d}({\bf x}) ,$$
where $\sigma_{d}(\bf x)$ is the normalized surface measure on
${\Bbb S}^{d}$ (i.e., $\sigma_{d}({\Bbb S}^{d})=1$). Denote by
$L_p\equiv L_p(\ss)$ the usual Lebesgue space with norm
$\|\cdot\|_{L_p}$. Then $L_2$ is a Hilbert space with   inner
product
$$\langle f,g\rangle=\int_{\ss}f(\x)g(\x)d\sz_{d}(\x).$$

Let $H$ be a reproducing kernel Hilbert space on $\ss$ with norm
$\|\cdot\|_H$. We want to approximate ${\rm INT}(f)$ from $f\in H$
using quadrature formulas
$$Q_{{\bf c},X_{n}}=\sum_{j=1}^{n}c_{j}f({\bf x}_{j})$$ with weights ${\bf
c}\in{\Bbb R}^{n}$ and knots $X_{n}=\{{{\bf x}_{1}},{{\bf
x}_{2}},\cdot \cdot \cdot, {{\bf x}_{n}}\}\subset {\Bbb S}^{d} $.

\
The worst case error of $Q_{c,X_{n}}$ is defined by
$$e(Q_{{\bf c},X_{n}},H,{\rm INT}):=\sup_{\|f\|_{H}\leq 1}|{\rm INT}(f)-Q_{{\bf c},X_{n}}(f)|,$$
and the $n$-th optimal  quadrature error
$$e_n(H,{\rm INT}):=\inf_{\sub{{\bf c}\in\Bbb R^n\\ X_n\subset\Bbb S^d}}e(Q_{{\bf c},X_{n}},H,{\rm INT})=\inf_{\sub{{\bf c}\in\Bbb R^n\\ X_n\subset\Bbb S^d}}\sup_{\|f\|_{H}\leq 1}|{\rm INT}(f)-Q_{{\bf c},X_{n}}(f)|.$$
We write $$e_n(H)=e_n(H,{\rm INT}).$$

We  are also interested in  approximation numbers and (linear)
sampling numbers in $L_2$. We  define the approximation numbers
$a_n(H)$ by
$$a_{n}(H)\equiv a_{n}(H,L_2) :=\inf_{\sub{L_1,...,L_n \in H^{'}\\ g_1,...,g_{n}\in L_2}}\sup_{\|f\|_{H}\leq 1}\|f-\sum_{i=1}^{n}L_{i}({f})g_{i}\|_{L_2},$$
and the (linear) sampling  numbers $g_n(H)$ by
$$g_{n}(H)\equiv g_{n}^{\rm lin}(H,L_2):=\inf_{\sub{{\bf x}_{i},...,{\bf x}_{n}\in
 {\Bbb S}^{d}\\ g_1,...,g_n\in L_2}}\sup_{\|f\|_{H}\leq 1}\|f-\sum_{i=1}^{n}f({\bf x}_{i})g_{i}\|_{L_2},$$
where $H'$ denotes the conjugate space of $H$ consisting of all
continuous linear functionals on $H$. The approximation numbers
and the (linear) sampling numbers  reflect the errors of the best
possible linear algorithms for $L_{2}$ approximation using at most
$n$ continuous linear functionals on $H$ or $n$ function values,
respectively. Clearly, we have
\begin{equation}\label{1.1} a_{n}(H)\leq g_{n}(H)\ \ \ {\rm and}\
\ \ e_{n}(H)\leq g_{n}(H).\end{equation} By \eqref{1.1} it
suffices to discuss lower bounds for $e_n(H)$ and upper bounds for
$g_n(H)$.

This paper is devoted to discussing the optimal quadrature error
and the sampling numbers for the Sobolev spaces ${\rm
H}^{\az,\beta}(\ss)$ with logarithmic perturbation (see the
definition in Section 2). Note that ${\rm H}^{\az,\beta}(\ss)$ is
a reproducing kernel Hilbert space if and only if
$\az>d/2,\beta\in\Bbb R$ or $\az=d/2, \beta>1/2$. Moreover,  this
space is just the Besov space $B_{2,2}^\Omega(\ss)$ (see
\cite{WT,WW1}) of generalized smoothness with equivalent norms,
where $\Omega(t)=t^\az(1+(\ln\frac1t)_+)^\beta$ for $t>0$,
$a_+=\max\{a,0\}$, see Subsection 2.1.

For the classical Sobolev spaces ${\rm H}^\az(\ss)={\rm
H}^{\az,0}(\ss)$ on $\ss$, upper and lower bounds for the optimal
quadrature error were obtained in \cite{BH} and \cite{H1},
asymptotics and strong equivalences of the approximation numbers
in \cite{K} and \cite{CW}, and asymptotics of the sampling numbers
in \cite{WW1}. For the Sobolev spaces ${\rm
H}^{\alpha,\beta}(\ss)$ with logarithmic perturbation, asymptotics
of the approximation numbers were obtained in \cite{WT} for
$\az>0$, asymptotics  of the sampling numbers in \cite{WW1} for
$\az>d/2$. Specifically, it was shown in \cite{WT} that for
$\az>0$,\begin{equation}\label{1.3-4} a_n({\rm
H}^{\az,\beta}(\ss))\asymp n^{-\az/d}\ln^{-\beta}n,\end{equation}
and in \cite{WW1} that for $\az>d/2$,
\begin{equation}\label{1.3-0} g_n({\rm H}^{\az,\beta}(\ss))\asymp
n^{-\az/d}\ln^{-\beta}n,\end{equation} where for positive
sequences $\{a_n\}, \{b_n\}$, $a_n \lesssim b_n$ means that there
is a constant $c> 0$ independent of $n$ such that $a_n \leq c b_n
$ for all $n\ge 2$. We write  $a_n \asymp b_n$ if $b_n \lesssim
a_n$ and $a_n \lesssim b_n $.

In the case $\az=d/2,\, \beta>1/2$ of small smoothness,
 Grabner and Stepanyuk  obtained in \cite{GS}, among others, upper and lower bounds for the
 optimal quadrature errors:
 $$n^{-1/2}\ln^{-\beta}n\lesssim
 e_n({\rm H}^{d/2,\beta}(\ss))\lesssim
 n^{-1/2}\ln^{-\beta+1/2}n.$$

  However, there is a logarithmic gap between the upper and lower
 bounds. They conjectured in \cite{GS} that the correct order for
 $e_n({\rm H}^{d/2,\beta}(\ss))$ is just the upper bound. Also
 Krieg and Vyb\'iral posed Open problem 2  in \cite{KV}, which stated that the existing lower bound can be improved.

There has been an increased interest in the comparison of standard
information given by function values and general linear
information for $L_2$ approximation problem. We refer to
\cite{DKU, KUV, KU1, KU2, NSU, T}
 for recent upper bounds and to \cite{HKNV1, HKNV2} for lower
 bounds. Specifically, applying random information,
weighted least squares method,  and infinite-dimensional variant
of subsampling strategy, the authors in \cite{DKU}  obtained  that
there are two absolute constants $c$, $C>0$, such that for every
reproducing
 kernel Hilbert space $H$ and every $n\in\Bbb N$, \begin{equation}\label{1.2}  g_n(H)\le
 C\Big(\sum_{k\ge cn}a_n(H)^2\Big)^{1/2}.\end{equation}
 By \eqref{1.3-4} and
\eqref{1.2} we have for $\beta>1/2$,
\begin{equation}\label{1.3} g_n({\rm H}^{d/2,\beta}(\ss))\lesssim
n^{-1/2}\ln^{-\beta+1/2}n.\end{equation}This give the upper bound
for $g_n({\rm H}^{d/2,\beta}(\ss))$. However, the algorithm used
in \cite{DKU} is not constructive (see \cite[Remark 7]{DKU}).

For constructive algorithms, it was shown in \cite{WW1} that the
filtered hyperinterpolation operators (see the definition in
\cite{WW1}) are asymptotically optimal linear algorithms for
$g_n({\rm H}^{\az,\beta}(\ss))$ with $\az>d/2$. Lu and Wang showed
in \cite{LW} that the weighted least squares operators (see
Subsection 2.3) are asymptotically optimal linear algorithms for
$g_n({\rm H}^{\az}(\ss))$ with $\az>d/2$.

In this paper, we investigate  the approximation numbers, the
sampling numbers, and the optimal quadrature error for the Sobolev
spaces ${\rm H}^{\az,\beta}(\ss),\, \az>0$ with generalized
smoothness. We make special attention to the case
$\az=d/2,\,\beta>1/2$.

 First, we give strong equivalences of the
approximation numbers $a_n({\rm H}^{\az,\beta}(\ss)),\ \az>0$
which generalized the result of \cite[Theorem 3.1]{CW}. Such
result gives a clue to Open problem 3 posed in \cite{KV}.

\begin{thm}\label{thm1.1}Let $\az>0,\beta\in\Bbb R$.
Then\begin{equation}\label{1.5} \lim_{n\to\infty}n^{\az/d}(\ln
n)^{\beta}a_n(
{\rm H}^{\az,\beta}(\ss))=\Big(\frac2{d\,!}\Big)^{\az/d}d^\beta.\end{equation}
\end{thm}

\begin{rem}If $\beta=0$, then Theorem \ref{thm1.1} reduces to \cite[Theorem 3.1]{CW}. One can rephrase \eqref{1.5} as strong equivalences
$$a_n(
{\rm H}^{\az,\beta}(\ss))\sim
\Big(\frac2{d\,!}\Big)^{\az/d}d^\beta n^{-\az}(\ln n)^{-\beta},$$
as $n\to\infty$  for arbitrary fixed $d$ and $\az>0,\beta\in\Bbb
R$. The novelty of Theorem \ref{thm1.1} is that it gives strong
equivalences of $a_n( {\rm H}^{\az,\beta}(\ss))$ and provides
asymptotically optimal constants, for arbitrary fixed $d$ and
$\az>0,\beta\in\Bbb R$. \end{rem}

Next, we discuss lower bounds for $e_n({\rm H}^{\az,\beta}(\ss))$
with $\az>d/2$ or $\az=d/2,\,\beta>1/2$. In the case of $\az>d/2$,
we can use  the  ``fooling" function technique, i.e., for any
quadrature rule $Q_{{\bf c},X_n}$, we can construct a fooling
function $f$ which vanishes in all nodes $X_n$, and satisfies
 $$\|f\|_{H^{\az,\beta}(\ss)}\lesssim 1,\ \ \ \ \int_{\ss}f(\x)d\sz_{d}(\x)\gtrsim
 n^{-\az/d}\ln^{-\beta}n.$$See \cite{DW, GS, W2} for the details of this
 technique. We use such technique in \cite{DW} to obtain the
 following theorem.
\begin{thm}\label{thm1.2} Let $\az>d/2$. Then
   \begin{equation} \label{1.6-0}e_n({\rm H}^{\az,\beta}(\ss))\gtrsim
n^{-\az/d}\ln^{-\beta}n. \end{equation}\end{thm}

 However,
in the case of $\az=d/2,\,\beta>1/2$, the  ``fooling" function
technique does not give  optimal lower bounds (see \cite{KV}).
 We  exploit the Hilbert space structure
and a variant of the Schur product theory to give optimal lower
  bounds  for $e_n({\rm H}^{d/2,\beta}(\ss))$  with $\beta>1/2$.
   Such technique was used in \cite{HKNV1, HKNV2, KV} on the $d$ dimensional  torus $\Bbb
   T^d$. It is very surprising that the above technique seems not to
   give optimal lower bounds  for $e_n({\rm H}^{\az,\beta}(\ss))$  with $\az>d/2,\ d\ge2$.

   \begin{thm}\label{thm1.3}Let $\az=d/2,\,\beta>1/2$. Then
   \begin{equation} \label{1.6}e_n({\rm H}^{\az,\beta}(\ss))\gtrsim
n^{-1/2}\ln^{-\beta+1/2}n.\end{equation}\end{thm}

By \eqref{1.1},  \eqref{1.6-0}, \eqref{1.6}, \eqref{1.3-0}, and
\eqref{1.3} we have the following corollary.

\begin{cor} \label{cor1.1} Let
 $\az>d/2$ or $\az=d/2,\,\beta>1/2$. Then
\begin{align} \label{1.7-0}  e_n({\rm H}^{\az,\beta}(\ss))\asymp
 g_n({\rm H}^{\az,\beta}(\ss))\asymp \bigg\{\begin{array}{ll} n^{-\az/d}\ln^{-\beta}n,\ \ & \az>d/2, \\
n^{-1/2}\ln^{-\beta+1/2}n, \ &\az=d/2, \beta>1/2.\end{array}
\end{align}\end{cor}

\begin{rem} By Corollary \ref{cor1.1}  and \eqref{1.3-4} we
obtain for $\az=d/2,\,\beta>1/2$,
\begin{align} \label{1.7} n^{-1/2}\ln^{-\beta+1/2}n\asymp e_n({\rm H}^{\az,\beta}(\ss))\asymp
 g_n({\rm H}^{\az,\beta}(\ss))\asymp
a_n({\rm H}^{\az,\beta}(\ss))\,\ln^{1/2}n.\end{align} This
 confirms the conjecture in \cite{GS} and solves Open
problem 2 posed in \cite{KV}. Of special interest is the
logarithmic gap between the sampling numbers and the approximation
numbers, which has been shown in \cite{HKNV2} on the torus $\Bbb
T$, in \cite{KV} on the $d$-dimensional torus $\Bbb T^d$, and is
now extended to the sphere $\ss$.\end{rem}

Finally, we give  constructive algorithms to obtain upper bounds
for $g_n({\rm H}^{\az,\beta}(\ss))$ with  $\az>d/2$ or
$\az=d/2,\,\beta>1/2$. The algorithms are the weighted least
squares algorithms investigated in \cite{Gr, LW}.

 An $L_2$ Marcinkiewicz-Zygmund (MZ) family $\mathcal{X}$  on a compact space
 $\mathcal M$
 is a double-indexed set of points $\{{\bf
x}_{n,k},\ n=1,2\dots, \,k=1,2,\dots, l_n\}$ in $\mathcal M$  with
a weight $\tau$ of positive numbers $\{\tau_{n,k}\}$ such that the
sampled $\ell_2$ norm of the $n$-th layer
$\left(\sum_{k=1}^{l_n}\tau_{n,k}|Q({\bf
x}_{n,k})|^2\right)^{1/2}$ is equivalent to the $L_2$ norm for all
``polynomials" $Q$ of degree at most $n$   with uniform constants.
Given an $L_2$ MZ family on $\mathcal M$, Gr\"ochenig in \cite{Gr}
introduced weighted least squares operators and least squares
quadrature rules, and derived approximation theorems and least
squares quadrature errors for Sobolev spaces. However, due to the
generality of the compact space
 $\mathcal M$, the obtained results are not optimal.
 Lu and Wang in \cite{LW}  proved that the weighted least
squares  operators $L_N^{\Bbb S}$ and the least squares quadrature
rules $I_N^{\Bbb S}$ on $\ss$ are asymptotical optimal for
$g_n({\rm H}^{\az}(\ss))$ and $e_n({\rm H}^{\az}(\ss))$ with
$\az>d/2$.

The following theorem shows that the weighted least squares
operators $L_N^{\Bbb S}$ and the least squares quadrature rules
$I_N^{\Bbb S}$ on $\ss$ are asymptotical optimal for $g_n({\rm
H}^{\az,\beta}(\ss))$ and $e_n({\rm H}^{\az,\beta}(\ss))$ with
$\az>d/2$ or $\az=d/2,\, \beta>1/2$.

\begin{thm}\label{thm1.5} Let $\mathcal{X}$ be an
$L_2$ MZ family on $\ss$ with associated weight $\tau$, global
condition number $\kappa$, and $l_N\asymp N^d\asymp n$, and let
$L_{N}^{\Bbb S}$ be the weighted least squares operator and
$I_N^{\Bbb S}$ be least squares quadrature rule induced by
$\mathcal{X}$.  We have for all $f\in {\rm H}^{\az,\beta}(\ss) $
with $\az>d/2$ or $\az=d/2, \beta>1/2$,
\begin{equation} \label{1.8-0}\left|\int_{\Bbb S^{d}}f({\bf x}){\rm d}\sigma_d ({\bf x})-I_N^{\Bbb S}(f)\right|\le \|f-L_{N}^{\Bbb S}(f)\|_{L_2}\le
c(1+\kappa^{1/2})C_{\az,\beta}(N)\|f\|_{{\rm
H}^{\az,\beta}(\ss)},\end{equation}  where $c>0$ is independent of
$n$,
$\kappa$, $f$, or $\mathcal{X}$, and $$C_{\az,\beta}(N):=\bigg\{\begin{array}{ll} N^{-\az}\ln^{-\beta}N,\ \ & \az>d/2, \\
N^{-\az}\ln^{-\beta+1/2}N, \ &\az=d/2, \beta>1/2.\end{array}$$
\end{thm}

\begin{rem} Let $\mathcal{X}$ be an
$L_2$ MZ family on $\ss$ with $l_N\le n$ and $l_N\asymp N^d\asymp
n$. For  such $L_2$ MZ family it follows from \eqref{1.7-0} and
\eqref{1.8-0} that
$$g_n({{\rm
H}^{\az,\beta}(\ss)})\le \sup_{\|f\|_{{\rm H}^{\az,\beta}(\ss)}\le
1}\|f-L_N^{\Bbb S}(f)\|_2\asymp g_n({{\rm
H}^{\az,\beta}(\ss)})\asymp C_{\az,\beta}(N),$$which implies that
the weighted least squares operators $L_N^{\Bbb S}$ are order
optimal linear algorithms in the sense of the sampling numbers.
Also, we have
$$e_n ({\rm
H}^{\az,\beta}(\ss))\le \sup_{\|f\|_{{\rm H}^{\az,\beta}(\ss)}\le
1} \left|\int_{\sph} f(\x) \, d\sz_d(\x)-I_N^{\Bbb
S}(f)\right|\asymp e_n ({\rm H}^{\az,\beta}(\ss))\asymp
C_{\az,\beta}(N),
$$ which means that  the least squares quadrature rules $I_N^{\Bbb S}$ are the
order optimal quadrature  formulas.

\end{rem}

The  paper is structured as follows. In Subsection 2.1 we
introduce Sobolev spaces $H^{\az,\beta}$ with logarithmic
perturbation. Then we delve into lower bounds for optimal
quadrature and positive definiteness and present some auxiliary
lemmas in Subsection 2.2. Definition and properties of the
weighted least squares operators are given in Subsection 2.3.
After that, in Section 3 we give the  proofs of  Theorems
\ref{thm1.1} and \ref{thm1.2}. Section 4 is devoted to elaborating
the lower bounds for  $e_n({\rm H}^{\az,\beta}(\ss))$ with
$\az=d/2,\, \beta>1/2$. We give the proof of Theorems \ref{thm1.5}
in Section 5.

\section{Preliminaries}

\subsection{Sobolev spaces ${\rm H}^{\az,\beta}(\ss)$ with logarithmic
perturbation}\

 We denote by $\mathcal{H}_\ell^d$ the space of
all spherical harmonics of degree $\ell$ on $\ss$. Denote by
$\Pi^{d+1}_m$  the set of spherical polynomials on $\mathbb{S}^d$
of degree $\le m$, which is just the set of polynomials of total
degree $\le m$ on $\mathbb{R}^{d+1}$ restricted to $\mathbb{S}^d$.
It is well known (see \cite[Corollaries 1.1.4 and 1.1.5]{DaiX})
that the dimension of $\mathcal{H}_\ell^d$ is
\begin{equation}\label{2.1} N(d,\ell):={\rm
dim}\,\mathcal{H}_\ell^d=\left\{\begin{array}{cl} 1,\ \ \
   & {\rm if}\ \ \ell=0,\\
\frac{(2\ell +d-1)\,(\ell +d-2)!}{(d-1)!\ \ell!}, \ \     & {\rm
if}\ \ \ell=1,2,\dots,
\end{array}\right.\end{equation} and \begin{equation}\label{2.2}C(d,m):={\rm
dim}\,\Pi_m^{d+1}=\frac{(2m+d)(m+d-1)!}{m!\,d!},\ m\in\Bbb
N.\end{equation} It is evident that $$N(d,\ell)\asymp \ell^{d-1}\
\ \ {\rm and} \ \ C(d,m)\asymp m^d.$$

 Let
$$\{Y_{\ell,k}\ |\ k=1,\dots, N(d,\ell)\}$$
be a collection of $L_2$-orthonormal real spherical harmonics of
degree $\ell$ (see, e.g., \cite{M}). Each spherical harmonic
$Y_{\ell,k}$ of exact degree $\ell$ is an eigenfunction of the
negative Laplace-Beltrami opertor $-\Delta_{0}^{*}$ with
eigenvalue
$$\widehat{\lz_\ell}=\ell(\ell+d-1).$$
By  the addition formula for spherical harmonic  of exact degree
$\ell$, we have  for $ \x,\y \in\ss,$
 \begin{equation}\label{2.3}
\sum_{k=1}^{N(d,\ell)}Y_{\ell,k}(\x)Y_{\ell,k}(\y)=\frac{n+\frac{d-1}2}{\frac{d-1}2}C^{\frac{d-1}2}_{n}(\x\cdot\y)=N(d,\ell)G_\ell(\x\cdot\y),
\end{equation}
where $C_{n}^{\frac{d-1}2}(t), \ n\ge 0,\,t\in[-1,1]$ are the
usual Gegenbauer polynomials defined as in \cite[ pp. 18-20 and
Proposition 1.4.11]{DX}, and $G_n,\,n=0,1,\dots,$ are the
Gegenbauer polynomials of the index $\frac{d-1}2$ normalized by
\begin{equation}\label{2.4} G_n(1)=1,\ \ \
C_n^{\frac{d-1}2}(1)\,G_n(t)=C_n^{\frac{d-1}2}(t).
\end{equation}
We fix $G_n$ in the sequel.

Let
$$\{Y_{\ell,k}\ |\ k=1,\dots, N(d,\ell)\}$$
be a fixed orthonormal basis for $\mathcal{H}_\ell^d$. Then
$$\{Y_{\ell,k}\ |\ \ell=0,1,2,\dots,\ k=1,\dots, N(d,\ell) \}$$ is
an orthonormal basis for the Hilbert space $L_2(\ss)$. Thus any
$f\in L_2(\ss)$ can be expressed by its Fourier (or Laplace)
series
$$f=\sum_{\ell=0}^{\infty}
H^d_\ell(f)=\sum_{\ell=0}^{\infty}{\sum_{k=1}^{N(d,\ell)} \langle
f,Y_{\ell,k}\rangle Y_{\ell,k}},$$
 where $$H_\ell^d(f)(\x)=\sum\limits_{k=1}^{N(d,\ell)} \langle
f,Y_{\ell,k}\rangle
Y_{\ell,k}(\x)=N(d,\ell)\int_{\ss}f(\y)G_\ell(\x\cdot
\y)d\sz_{d}(\y)$$ denote the orthogonal projections of $f$ onto
$\mathcal{H}_\ell^d$, and
$$\langle f,Y_{\ell,k}\rangle =\int_{\ss} f(\x)
Y_{\ell,k}(\x) \, d\sz_{d}(\x)$$ are the  Fourier coefficients of
$f$. We also have the following Parseval equality,
$$\|f\|_{L_2}=\Big(\sum_{\ell=0}^{\infty}\sum_{k=1}^{N(d,\ell)} |\langle
f,Y_{\ell,k}\rangle
|^2\Big)^{1/2}=\Big(\sum_{\ell=0}^{\infty}\|H_\ell^d(f)\|_{L_2}^2\Big)^{1/2}.$$

Let $S_n$ denote the orthogonal projection  onto $\Pi_n^{d+1}$.
Then for $f\in L_1(\ss)$,
\begin{equation}\label{2.5-0}
S_n(f)=\sum_{j=0}^nH_j^d(f)=\int_{\ss}f(\y)E_n(\x,\y)d\sz_d(\y),
\end{equation}where $$E_n(\x,\y)=\sum_{\ell=0}^n \sum_{k=1}^{N(d,\ell)}
Y_{\ell,k}(\x) Y_{\ell,k}(\y)$$ is the reproducing kernel of
$\Pi_n^{d+1}$ with respect to the inner product of $L_2$.

For $f\in L_p,\ 1\le p\le \infty,$ we define
$$E_n(f)_{L_p}=\inf_{g\in\Pi_n^{d+1}}\|f-g\|_{L_p}.$$
It is well known that
\begin{equation}\label{2.5}
E_n(f)_{L_2}=\|f-S_n(f)\|_{L_2}. \end{equation}

 Let
$\Lz=\{\lz_\ell\}_{\ell=0}^\infty$ be a bounded sequence, and let
$T^{\Lz}$ be a multiplier operator on $L_2(\ss)$ defined by
$$T^{\Lz} (f)=\sum_{\ell=0}^{\infty}\lz^{-1}_\ell
H^d_\ell(f)=\sum_{\ell=0}^{\infty}\lz^{-1}_\ell
{\sum_{k=1}^{N(d,\ell)} \langle f,Y_{\ell,k}\rangle Y_{\ell,k}}.$$

\begin{defn}\label{d2.1}Let $\Lz=\{\lz_\ell\}_{\ell=0}^\infty$ be
a non-increasing positive  sequence with
$\lim\limits_{\ell\to\infty}\lz_\ell=0$. We define the multiplier
space ${\rm H}^{\Lz}(\ss)$ by
\begin{align*}{\rm H}^{\Lz}(\ss)&:=\Big\{T^{\Lz} f\,    \big|\, f\in
L_2(\ss) \ ,\ \|T^{\Lz}
f\,\|_{{\rm H}^{\Lz}(\ss)}=\|f\|_{L_2(\ss)}<\infty\Big\}\\&
 := \Big\{f\in
L_2(\ss)\,    \big|\,
\|f\,\|_{{\rm H}^{\Lz}(\ss)}:=\Big(\sum_{\ell=0}^{\infty}\lz_\ell^{-2}\sum_{k=1}^{N(d,\ell)}
|\langle f,Y_{\ell,k}\rangle|^2\Big)^{1/2}<\infty\Big\}.
\end{align*}\end{defn}

Clearly, the multiplier  space ${\rm H}^{\Lz}(\ss)$ is a Hilbert space
with inner product
$$\langle f,g \rangle_{{\rm H}^{\Lz}(\ss)}= \sum_{\ell=0}^{\infty}\lz_\ell^{-2}\sum_{k=1}^{N(d,\ell)}
\langle f,Y_{\ell,k}\rangle\, \langle g,Y_{\ell,k}\rangle=
\sum_{\ell=0}^{\infty}\lz_\ell^{-2} \langle H_\ell^d
(f),H_\ell^d(g)\rangle,$$ and
\begin{equation}\label{2.5-1}
E_n(f)_{L_2}=\|f-S_n(f)\|_{L_2}\le \lz_{n} \|f\|_{{\rm
H}^{\Lz}(\ss)}.
\end{equation}

Moreover, we remark that the embedding ${\rm
H}^{\Lz}(\ss)\hookrightarrow C(\ss)\ ({\rm or}\ L_{\infty}(\ss))$
is continuous if
\begin{equation}\label{2.6}
\sum_{\ell=0}^{\infty}\lz^{2}_\ell N(d,\ell)<\infty,
\end{equation}
which ensures ${\rm H}^{\Lz}(\ss)$ is a reproducing kernel Hilbert space
with the reproducing kernel
\begin{align}\label{2.7}
  K^\Lambda({\bf x},{\bf y})
  &=\sum_{\ell=0}^\infty\lz_\ell^2\sum_{k=1}^{N(d,\ell)}Y_{\ell,k}({\bf x})Y_{\ell,k}({\bf y})
    =\sum_{\ell=0}^\infty\lambda_\ell^2N(d,\ell)G_\ell({\bf x}\cdot{\bf
    y}),
\end{align}where $G_\ell(t),\ t\in
[-1,1]$ is given in \eqref{2.4}. Indeed, if \eqref{2.6} holds, then  we have
 for any
$\x\in\ss$,
$$\|K^\Lz(\x,\cdot)\|_{{\rm H}^\Lz(\ss)}^2=\langle K^\Lz(\x,\cdot),
 K^\Lz(\x,\cdot)\rangle_{{\rm H}^\Lz(\ss)}=K^\Lz(\x,\x)=\sum_{\ell=0}^\infty \lz_\ell^2 N(d,\ell)<\infty,
 $$and for $f\in {\rm H}^{\Lz}(\ss)$, $$\langle f,
K^\Lz(\x,\cdot)\rangle_{{\rm H}^\Lz(\ss)}=\sum_{\ell=0}^\infty\lz_\ell^2\sum_{k=1}^{N(d,\ell)}\langle
f,Y_{\ell,k}\rangle\lz_\ell^{-2}Y_{\ell,k}({\bf x})=f({\bf x}),$$
which means that ${\rm H}^{\Lz}(\ss)$ is a reproducing kernel Hilbert
space with the kernel $K^\Lz(\cdot,\cdot)$. It follows from the
Cauchy inequality that for $f\in {\rm H}^{\Lz}(\ss),\ \x,\y\in\ss,$
\begin{align*}|f(\x)-f(\y)|^2&=|\langle f,
K^\Lz(\x,\cdot)-K^\Lz(\y,\cdot)\rangle_{{\rm H}^\Lz(\ss)}|^2\\ &\le
\|f\|_{{\rm H}^\Lz(\ss)}^2\
\|K^\Lz(\x,\cdot)-K^\Lz(\y,\cdot)\|_{{\rm H}^\Lz(\ss)}^2\\
&=2\|f\|_{{\rm H}^\Lz(\ss)}^2 \sum_{\ell=0}^\infty
\lz_\ell^{2}N(d,\ell)(1-G_\ell(\x\cdot \y)).\end{align*}This
implies that $f$ is continuous on $\ss$. Hence, the space ${\rm
H}^{\Lz}(\ss)$ is continuously embedded in $C(\ss)$.

For $\Lz=\{\lz_\ell\}_{\ell=0}^\infty$, we set
$$\tilde\Lz=\{\tilde \lz_\ell\}_{\ell=0}^\infty,\ \
\tilde\lz_\ell=\lz_\ell^2N(d,\ell),\ \ell=0,1,\dots.$$ If
\eqref{2.6} holds, then $\tilde \Lz\in\ell_1(\Bbb N_0)$, i.e.,
$$\|\tilde \Lz\|_{\ell_1}\equiv\|\tilde \Lz\|_{\ell_1(\Bbb N_0)}:=\sum_{\ell=0}^\infty |\tilde
\lz_\ell|<\infty,
$$and the reproducing kernel of ${\rm H}^\Lz(\ss)$ can be written by
$$K^\Lz(\x,\y)=\sum_{\ell=0}^\infty \tilde \lz_\ell
G_\ell(\x\cdot\y),\ \ \x,\y\in\ss.$$ We write $${\rm
H}_{\widetilde\Lambda}:={\rm H}^{\Lambda}(\ss),\ \ \
K_{\widetilde\Lambda}({\bf x},{\bf y}):=K^\Lambda({\bf x},{\bf
y}).$$
\begin{defn}Let $\az>0,\,\beta\in\Bbb R$. The Sobolev space ${\rm H}^{\az,\beta}(\ss)$ with logarithmic
perturbation is defined to be the multiplier space
${\rm H}^{\Lz^{(\az,\beta)}}(\ss)$, where
$\Lz^{(\az,\beta)}=\{\lz_\ell^{(\az,\beta)}\}_{\ell=0}^\infty$,
and
$$\lz_\ell^{(\az,\beta)}=(1+\ell(\ell+d-1))^{-\az/2}(\ln(3+\ell(\ell+d-1)))^{-\beta}.
$$ \end{defn}

If $\beta=0$, then  the space ${\rm H}^{\az,\beta}(\ss)$ reduces
to the   Sobolev space ${\rm H}^{\az,+}$ as defined in
\cite[Definition 2.2]{CW}. Hence ${\rm H}^{\az,\beta}(\ss)$  is
the Sobolev space with logarithmic perturbation. Such spaces were
investigated intensively (see \cite{DT, GS, Mou, WW1}). In the
case $\az>d/2$ or $\az=d/2,\, \beta>1/2$, the space ${\rm
H}^{\az,\beta}(\ss)$ is continuously embedded into the space of
continuous functions $C(\ss)$, and   is a reproducing kernel
Hilbert space.

 For $f\in {\rm H}^{\az,\beta}(\ss),\, \az>0$, we  define
$$\sigma_1(f)=S_{2}(f), \ \ \sigma_j(f)=S_{2^{j}}(f)-S_{2^{j-1}}(f),\ \
{\rm for}\ j\ge2.$$Then $\sigma_j(f)\in\Pi_{2^{j}}^{d+1}$ and
$$f=\sum_{j=1}^\infty\sigma_j(f)$$ converges  in the $L_2$ norm.
Note that $\sigma_j(f)$ satisfies the Bernsetin inequality,
i.e., for any $r>0$,
$$ \|(-\triangle_0)^{r/2}(\sz_j(f))\|_{L_2}\lesssim
2^{jr}\|\sz_j(f)\|_{L_2}.$$ Also,  by \eqref{2.5} we obtain
\begin{align*}
  \|\sigma_j(f)\|_{L_{2}}\le
  \|f-S_{2^{j}}(f)\|_{L_{2}}+\|f-S_{2^{j-1}}(f)\|_{L_{2}}\le
  2E_{2^{j-1}}(f)_{L_{2}}.
  \end{align*}
Now set $\Omega(t)=t^\az(1+(\ln\frac1t)_+)^\beta$ for $t>0$.
Consider
  the Besov space $B_{2,2}^\Omega(\ss)$  of generalized smoothness  (see \cite{WT,WW1}). It
follows from \cite[Remark 2]{WT} and the above two inequalities
that $$\|f\|_{B_{2,2}^\Oz(\ss)}\asymp \Big(\sum_{j=0}^\infty \frac
{\|\sz_j(f)\|_{L_2}^2}{\Oz(2^{-j})^2}\Big)^{1/2}.
$$ On the other hand, we have
$$\|f\|_{{\rm H}^{\az,\beta}(\ss)}^2=\sum_{\ell=0}^\infty
(\lz_\ell^{(\az,\beta)})^{-2}\|H_\ell^d(f)\|_{L_2}^2=:\sum_{j=1}^\infty
I_j.
$$
Here, $$I_1:=\sum_{\ell=0}^2
(\lz_\ell^{(\az,\beta)})^{-2}\|H_\ell^d(f)\|_{L_2}^2\asymp
\sum_{\ell=0}^2 \|H_\ell^d(f)\|_{L_2}^2=\|\sz_1(f)\|_{L_2}^2\asymp
\frac{ \|\sz_1(f)\|_{L_2}^2}{\Oz(2^{-1})^2},
$$
and for $j\ge 2$,
$$I_j:=\sum_{\ell=2^{j-1}+1}^{2^j}
(\lz_\ell^{(\az,\beta)})^{-2}\|H_\ell^d(f)\|_{L_2}^2\asymp
2^{2j\az}j^{2\beta}\sum_{\ell=2^{j-1}+1}^{2^j}
\|H_\ell^d(f)\|_{L_2}^2\asymp \frac{
\|\sz_j(f)\|_{L_2}^2}{\Oz(2^{-j})^2}.
$$
It follows that
\begin{align}\|f\|_{{\rm H}^{\az,\beta}(\ss)}&\asymp \|f\|_{B_{2,2}^\Oz(\ss)}\asymp \Big(\sum_{j=1}^\infty \frac
{\|\sz_j(f)\|_{L_2}^2}{\Oz(2^{-j})^2}\Big)^{1/2}\notag\\&
\asymp\Big(\sum_{j=1}^\infty 2^{2j\az}
j^{2\beta}\|\sz_j(f)\|_{L_2}^2\Big)^{1/2}.\label{2.8}
\end{align} This means that the space ${\rm H}^{\az,\beta}(\ss)$ is
just the Besov space $B_{2,2}^\Omega(\ss)$ (see \cite{WT,WW1}) of
generalized smoothness  with above equivalent norms. Here,
$\Omega(t)=t^\az(1+(\ln\frac1t)_+)^\beta$ for $t>0$. Hence, from
\cite{WT,WW1}) we have
\begin{align}\|f\|_{{\rm H}^{\az,\beta}(\ss)}\asymp \|f\|_{L_2} +\Big(\sum_{j=1}^\infty 2^{2j\az}
j^{2\beta} E_{2^j}(f)_{L_{2}}^2\Big)^{1/2}.\label{2.8-0}
\end{align}

\subsection{Lower bounds for optimal quadrature and positive definiteness}\

Let $H \equiv H(K)$ be a reproducing kernel Hilbert space (RKHS)
consisting of real-valued functions on $\ss$ with  reproducing
kernel $K\in C({\Bbb S}^{d}\times {\Bbb S}^{d})$ and inner product
$\langle \cdot,\cdot\rangle_H$. That is, $H$ is a Hilbert space of
real-valued functions on $\Bbb S^d$ such that point evaluation
$$\delta_{\bf x}:\,H(K)\to\Bbb R,\ \ \delta_{\bf x}(f):=f({\bf x})$$
is a continuous functional on $H$ for all ${\bf x}\in\Bbb S^d$,
and the reproducing kernel $K$ satisfies

(i) for all $\x,\y\in\ss$, $K(\x,\y)=K(\y,\x)$;

(ii) for all $\x\in\ss$, $K(\x,\cdot)\in H$;

(iii) for all $\x\in\ss$ and $f\in H$, $f(\x)=\langle f,
K(\x,\cdot)\rangle_{H}$.

\noindent We refer to \cite{A, BT} for basics on RKHSs.

For $h\in H$, we define $S_h$ the continuous linear functional on
$H$ by  $S_h(f)=\langle f,h\rangle_{H}$ for $f\in H$. Since $\rm
INT$ is continuous on $H$, by the Riesz representation theorem,
there exists an $h_0 \in H$ for which
$${\rm INT}(f)=\int_{{\Bbb S}^{d}}f({\bf{x}})\, d\sigma_{d}({\bf
{x}})=\langle f,h_0\rangle_{H}=S_{h_0}(f).$$ We have for
$\x\in\ss$,
$$ h_0({\bf x})=\big\langle h_0, K({\bf x},\cdot) \big\rangle_{H}={\rm INT}(K({\bf x},\cdot))
=\int_{{\Bbb S}^{d}}K({\bf{x}},{\bf{y}}) d\sigma_{d}({\bf y}).$$

 We consider the space ${\rm H}^{\Lz}(\ss)={\rm H}_{\tilde \Lz}$, where $\Lz=\{\lz_k\}_{k=0}^\infty$ is a positive
 sequence, $\tilde \Lz=\{\tilde \lz_k\}_{k=0}^\infty$, and $\tilde\lz_k=\lz_k^2N(d,k),\ k=0,1,\dots$. As seen in Subsection 2.1, the space ${\rm H}^{\Lz}(\ss)$
 is  RKHS on the sphere  if and only if $$\sum_{k=0}^\infty
 \lz_k^2N(d,k)=\sum_{k=0}^\infty\tilde\lz_k<+\infty.$$
 In this case, the reproduce kernel is $$K_{\tilde{\Lz}}({\bf x},{\bf y}):=\sum_{\ell=0}^\infty \widetilde{\lambda}_\ell G_\ell ({\bf x}\cdot{\bf y}),
\  {\bf x},{\bf y}\in\Bbb S^d,$$and the representer of the
integration problem INT is
$$ h_{\widetilde \Lz}({\bf x})=\int_{{\Bbb S}^{d}}K_{\widetilde \Lz}({\bf{x}},{\bf{y}}) d\sigma_{d}({\bf y})
=\sum_{\ell=0}^\infty \widetilde{\lambda}_\ell \int_{\Bbb
S^d}G_\ell ({\bf x}\cdot{\bf y})d\sigma_{d}({\bf
y})=\widetilde{\lambda}_0=\lz_0^2,$$ for all ${\bf{x}}\in {\Bbb
S}^{d} $. This means that
$${\rm INT}(f)=\int_{{\Bbb S}^{d}}f({\bf{x}})\, d\sigma_{d}({\bf
{x}})=\langle f,h_{\widetilde \Lz}\rangle_{{\rm
H}_{\tilde\Lz}}=S_{h_{\widetilde \Lz}}(f).$$ Furthermore, we have
$$\|h_{\widetilde \Lz}\|_{{\rm H}_{\tilde\Lz}}=(\tilde\lambda_0)^{1/2}=\lz_0.$$

 We say that a symmetric matrix $M\in {\Bbb R}^{n\times n}$
is positive semi-definite if $\mathbf{c}^{T}M\mathbf{c} \geq 0$
for all $\mathbf{c}\in {\Bbb R}^{n}$. If $M,N \in {\Bbb
R}^{n\times n}$, we denote by $M\circ N$ their Hadamard (Schur)
product, i.e., a matrix with entries $(M\circ
N)_{j,k}=M_{j,k}\cdot N_{j,k}$ for all $j,k=1,...,n$ (see
\cite{H}). Furthermore, the partial ordering $M\succeq N$ of
symmetric  matrices means that $ M-N $ is positive semi-definite.
We denote by ${\rm diag }M=(M_{1,1},...,M_{n,n})^{T}$ the vector
of diagonal entries of $M$ whenever $M\in {\Bbb R}^{n\times n}$.

In order to prove  lower bounds for $e_n({\rm H}_{\tilde \Lz})$ we
need the following lemmas. The first lemma is Schoenberg's
remarkable result about positive definite functions on the sphere.

\begin{lem}(See \cite{Sch, DaiX}.) \label{lem2.3} Let $\Gamma=\{\gz_j\}_{j=0}^\infty\in\ell_1(\Bbb N_0)$, and
let
$$f_\Gz (t)=\sum_{j=1}^\infty \gz_j G_j(t),\ \  \ \ K_\Gz(\x,\y)=f_\Gz(\x\cdot\y),$$where $G_j$ are the
normalized Gegenbauer polynomials given in \eqref{2.4}, $t\in
[-1,1]$, $\x,\y\in\ss$. Then for all $n\in \Bbb N$ and ${\bf
x}_{1},{{\bf x}_{2}},\cdot \cdot \cdot, {{\bf x}_{n}} \in {\Bbb
S}^{d} $, the matrix
$$\big(K_\Gz(\x_{j},\x_{k})\big)_{1\le
j,k\leq n}$$ is positive semi-definite if and only if $\Gz$ is
nonnegative.\end{lem}

The second lemma gives a characterization of lower bounds for
$e_n(H)$ via the positive definiteness of certain matrices
involving the reproducing kernel $K$ and the representer $h_0$ of
the integration problem INT.

\begin{lem}(See \cite[Proposition 1]{HKNV1} and \cite{HKNV2, KV}.)
\label{lem2.4} Let $H$ be a RKHS on $\ss$ with reproducing kernel
$K \in C(\ss \times \ss)$, and let ${\rm INT}=\langle \cdot, h_0
\rangle_{H}$ for some $h_0 \in H$,  $\alpha > 0$. For all $n\in
\Bbb N$ and ${\bf x}_{1},{{\bf x}_{2}},\cdot \cdot \cdot, {{\bf
x}_{n}} \in {\Bbb S}^{d} $, the matrix
$$\big(K(\x_{j},\x_{k})- \alpha h_0(\x_{j})h_0(\x_{k})\big)_{1\le
j,k\leq n}$$ is positive semi-definite, if and only if
$$e_{n}^{2}(H)\geq \|h_0\|_{H}^{2}- \alpha^{-1}.$$
\end{lem}

The third lemma was obtained by J. Vyb\'iral in  \cite{V} and
shows
 that the Schur product of a positive semi-definite matrix with itself
 is not only positive semi-definite but also bounded from below by some rank-1 matrix.

\begin{lem}(See \cite[Theorem 1]{V}.) \label{lem2.5}
Let $M\in {\Bbb R}^{n\times n}$ be a  positive semi-definite
matrix. Then
$$M\circ M \succeq \frac{1}{n}({\rm diag} \,M)({\rm diag} \, M)^{T}.$$

\end{lem}

\subsection{Weighted least squares operators and least squares quadrature rules
}\

First we give the definition of $L_2$ MZ families.
\begin{defn}Let $\mathcal X  =\{\mathcal X_n\}= \{\x_{n,k},\ n=1,2,\dots, k = 1,\dots,
l_n\}$ be a doubly-indexed set of points in $\ss$ and $\tau=
\{\tau_{n,k} : n=1,2,\dots, \, k = 1,\dots, l_n\}$ be a family of
positive numbers. Then $\mathcal X $ is called an $L_2$
Marcinkiewicz-Zygmund (MZ) family with associated weight $\tau$,
if there exist constants $A, B
> 0$ independent of $n$ such that \begin{equation}\label{2.33-1}A\|p\|_2^2\le \sum_{k=1}^{l_n}|p(\x_{n,k})|^2\tau_{n,k}\le B\|p\|_2^2\ \qquad {\rm
for\ all}\ p\in\Pi_n^{d+1}. \end{equation} The ratio $\kappa =
B/A$ is the global condition number of the $L_2$ MZ family
$\mathcal X$, and $\mathcal X _n = \{\x_{n,k} : k = 1, \dots,
l_n\}$ is the $n$th layer of $\mathcal X$.\end{defn}

The existence of $L_2$ MZ families on $\ss$ (including the case
$\kappa=1$) has been established in \cite{MNW, BD, DaiX, MO, MP}.
Additionally, necessary or sufficient density conditions for $L_2$
MZ families on $\ss$ have been discussed in \cite{MP, M}.

 Let
$\mathcal X$ be an $L_2$ MZ family with associated weight $\tau$.
Given the samples $\{ f (\x_{N,k} )\}$ of a continuous function
$f$ on $\ss$ on the $N$th layer $\mathcal X_N$, Gr\"ochenig in
\cite{Gr} introduced the weighted least squares polynomials
$L_N^{\Bbb S}(f)$ to approximate $f$ using only these samples,
i.e., the  solution of weighted least squares problems:
\begin{equation}\label{2.33-2}L_{N}^{\Bbb S}(f) =\arg \min_{p\in\Pi_N^{d+1}}\
\sum_{k=1}^{l_N}|f(\x_{N,k})-p(\x_{N,k})|^2\tau_{N,k}.\end{equation}
This procedure yields  a sequence $\{L_{N}^{\Bbb S}(f)\}$ of the
best weighted $\ell_2$-approximation of the data $\{
f(\x_{N,k})\}$ by   a
  spherical polynomial in $\Pi_N^{d+1}$ for every $f\in
C(\ss)$. We call  $L_{N}^{\Bbb S}$   the weighted least squares
operator.

It follows from \eqref{2.33-1} that for $p\in \Pi_N^{d+1}$, $p=0$
whenever $p(\x_{N,k})=0$. This means that usually $\mathcal X_N$
contains more than $d_N=\dim \Pi_N^{d+1}$ points, so that it is
not an interpolating set for $\Pi_N^{d+1}$. Therefore,
$L_{N}^{\Bbb S}(f)$ is usually  a quasi-interpolant.

We use the weighted discretized inner product
\begin{equation}\label{4.2.20}\langle
f,g\rangle_{(N)}:=\sum_{k=1}^{l_N}f(\x_{N,k})g(\x_{N,k})\tau_{N,k}\end{equation}
and the discretized norm $$\|f\|_{(N)}^2=\langle
f,f\rangle_{(N)}.$$

We consider the corresponding orthogonal polynomial projection
$L_N^{\Bbb S}$ onto $\Pi_N^{d+1}$ with respect to the weighted
discretized inner product $\langle \cdot,\cdot\rangle_{(N)}$,
namely the weighted least squares polynomial $ L_{N}^{\Bbb S}(f)$
defined by
$$L_{N}^{\Bbb S} (f)=\arg\min_{p\in\Pi_N^{d+1}} \|f-p\|_{(N)}^2
=\arg \min_{p\in\Pi_N^{d+1}}\
\sum_{k=1}^{l_N}|f(\x_{N,k})-p(\x_{N,k})|^2\tau_{N,k}.$$ We shall
give  a formal   expression for  $L_{N}^{\Bbb S}(f)$ (see
\cite{LW}). Let $\varphi_i,\ i=1,\dots, d_N$ be an orthonormal
basis of $\Pi_N^{d+1}$ with respect to the weighted discrete
scalar product \eqref{4.2.20}, i.e.,
$$\langle
\varphi_i,\varphi_j\rangle_{(N)}=\delta_{i,j}=\Big\{\begin{array}{cl}
1,\ \
   & {\rm if}\ \ i=j,\\
0,\ \     & {\rm if}\ \ i\neq j,
\end{array} \ \ \  1\le i,j\le d_N.$$ We set
$$D_N(\x,\y)=\sum_{k=1}^{d_N}\varphi_k(\x)\varphi_k(\y).$$
Clearly, $D_N(\x,\y)$ is the reproducing kernel of $\Pi_{N}^{d+1}$
with respect to the weighted discrete scalar product
\eqref{4.2.20}, and the weighted  least squares polynomial
$L_{N}^{\Bbb S}(f)$ is just the orthogonal projection of $f$ onto
$\Pi_{N}^{d+1}$  with respect to the weighted discrete scalar
product \eqref{4.2.20}, i. e.,
\begin{equation}\label{4.2.21}L_{N}^{\Bbb S}(f)({\bf x})=\langle
f,D_N(\cdot,\x)\rangle_{(N)}=\sum_{k=1}^{l_N}
f(\x_{N,k})D_N(\x_{N,k},\x)\tau_{N,k}.\end{equation}

Obviously, $L_N^{\Bbb S}$ is  a linear operator satisfying
\begin{align}\label{2.33-3}
 \|f-L_{N}^{\Bbb S}(f)\|_{(N)}\le \|f-p\|_{(N)},\ \ {\rm for\ all}\ p\in\Pi_N^{d+1}.
\end{align}

Also, Gr\"ochenig  in \cite{Gr} used the frame theory
  to construct   \emph{least squares  quadrature rules}
$$
I_N^{\Bbb S}(f)=\sum_{k=1}^{l_N}w_{N,k}f({\bf x}_{N,k})
$$ induced by an $L_{2}$ MZ
  family.  Lu and Wang in \cite{LW} proved that
$$w_{N,k}=\tau_{N,k}\int_{\ss}D_N({\bf x},{\bf x}_{n,k})d\sz_d({\bf x}), \ \ \ I_N^{\Bbb S}(f)=\int_{\ss}L_{N}^{\Bbb S}(f)({\bf x}){d}\sz_d({\bf x}), $$
and
\begin{equation}\label{2.33-4}\left|\int_{\ss}f({\bf x}){ d}\sz_d({\bf x})-I_N^{\Bbb S}(f)\right|\le \left\|f-L_{N}^{\Bbb S}(f)\right\|_{L_2}.\end{equation}

Finally, we give the relation between  weighted least squares
polynomials and hyperinterpolation (see \cite{LW}).
Hyperinterpolation
 was originally introduced by I. H. Sloan (see \cite{S}).
It uses the Fourier orthogonal projection of a function which can
be expressed in the form of integrals, but approximates the
integrals used in the expansion by means of a positive quadrature
formula. In recent years, hyperinterpolation and its
generalization on $\ss$ has attracted much interest, and a great
number of interesting results
 have been obtained (see \cite{D2,  GS1,   HS,  R, R2, S, SW0, SW1, W2,  WS}).

  Assume that
$Q_N(f)=\sum\limits_{k=1}^{l_N}\tau_{N,k}f(\x_{N,k}),\
N=1,2,\dots$ is a sequence of positive quadrature formulas on
$\ss$ which are exact for $\Pi_{2N}^{d+1}$, i.e., $\tau_{N,k}>0$,
and for all $f\in \Pi_{2N}^{d+1}$,
$$\int_{\ss}f(\x)d\sz_d(\x)=Q_N(f)=\sum_{k=1}^{l_N}\tau_{N,k}f(\x_{N,k}).$$ Then
the family $\mathcal X=\{\x_{N,k}\}$ with weight $\tau$ is an $L_2$ MZ family with the
constants $A=B=1$ and the global condition number $\kappa$
 is equal to $1$. On the other hand, if the global condition number $\kappa$ of an $L_2$ MZ
family $\mathcal X=\{\x_{N,k}\}$ with  weight $\tau$ is equal to
$1$, then $\mathcal X$ determines a sequence of positive
quadrature formulas $Q_N(f)=\frac
1A\sum\limits_{k=1}^{l_N}\tau_{N,k}f(\x_{N,k})$ on $\ss$ which are
exact for $\Pi_{2N}^{d+1}$.

Let $Q_N(f)=\sum\limits_{k=1}^{l_N}\tau_{N,k}f(\x_{N,k}),\
N=1,2,\dots$ be a sequence of positive quadrature formulas on
$\ss$ which are exact for $\Pi_{2N}^{d+1}$.  The
hyperinterpolation operator $H_N^{\Bbb S}$ on $\ss$ is defined by
\begin{equation*}H_{N}^{\Bbb S}(f)(\x)=\sum\limits_{k=1}^{l_N}\tau_{N,k}f(\x_{N,k})E_{N}(\x,\x_{N,k}),\
\ \ f\in C(\ss), \end{equation*} where $E_N(\x,\y)$ is the
reproducing kernel of $\Pi_N^{d+1}$ with respect to the inner
product of $L_2$. We note that for all $f,g\in\Pi_N^{d+1}$,
$$\langle f, g\rangle=\langle f,g\rangle_{(N)},$$
where $\langle \cdot,\cdot\rangle_{(N)}$ is defined by
\eqref{4.2.20}. This means that $$D_N(\x,\y)=E_N(\x,\y),\ \ \
H_N^{\Bbb S}(f)=L_N^{\Bbb S}(f),$$and
$$ I_N^{\Bbb S}(f)=Q_N(f).$$

Hence,  the hyperinterpolation
 is just  the weighted least squares polynomial for an $L_2$  MZ
family with the global condition number $\kappa=1$, and a sequence
of positive quadrature formulas $Q_{N}(f)$ on $\ss$ which are
exact for $\Pi_{2N}^{d+1}$ is just the least squares quadrature
for an $L_2$  MZ family with the global condition number
$\kappa=1$, and the weighted least squares polynomial and the
least squares quadrature may be viewed as a generalization of the
hyperinterpolation and a sequence of positive quadrature formulas
$Q_{N}(f)$ which are exact for $\Pi_{2N}^{d+1}$, respectively.

\section{Proofs of Theorems \ref{thm1.1} and \ref{thm1.2}}

Theorem \ref{thm1.1} follows from the following theorem
immediately.

\begin{thm}\label{thm3.1}
Let $\Lz=\{\lz_{\ell,d}\}_{\ell=0}^\infty$ be  a  nonnegative non-increasing
 sequence with $\lim\limits_{\ell\to\infty}\lz_{\ell,d}\,
\ell^{\az} \ln^{\beta} \ell=1$ for some $\az>0$ and $\beta\in \Bbb
R$. We have
\begin{equation}\label{3.1}\lim_{n\to\infty}n^{\az/d}(\ln n)^{\beta}a_n({\rm H}^{\Lz}(\ss))=\Big(\frac2{d\,!}\Big)^{\az/d}d^\beta.\end{equation}
\end{thm}

In order to prove Theorem \ref{thm3.1} we need the following two
lemmas.
\begin{lem} (See \cite[Theorem 2.8]{CW}.) \label{lem3.2-0} Let $\Lz=\{\lz_{\ell,d}\}_{\ell=0}^\infty$
be  a nonnegative non-increasing
 sequence with
$\lim\limits_{\ell\to\infty}\lz_{\ell,d}=0$. For $C(d,\ell-1)<n\le
C(d,\ell),\ \ell=0, 1,2,\dots,$ we have
$$a_n({\rm H}^\Lz(\ss)))=\lz_{\ell,d},
$$where we set $C(d,-1)=0$. \end{lem}

\begin{lem} \label{lem3.3-0} Let $\{a_n\},\ \{b_n\}$ be two positive sequences such
that
$$\lim_{n\to\infty}\frac{b_n}{a_n}=A\neq 0\ \ {\rm and}\ \  \lim_{n\to\infty}a_n=+\infty.$$
Then we have $$\lim_{n\to\infty}\frac{\ln b_n}{\ln a_n}=1.$$\end{lem}
\begin{proof}
According to $\lim\limits_{n\to\infty}\frac{b_n}{a_n}=A$, we can
get for arbitrary $\varepsilon\in(0,A)$, there is a positive
integer $N$ such that $n> N$,
$$\big|\frac{b_n}{a_n} - A \big|< \varepsilon.$$
This means that  $$\ln a_n+\ln(A - \varepsilon) < \ln {b_n} < \ln
a_n+\ln(A + \varepsilon) ,$$which yields that $$1+\frac{\ln((A-
\varepsilon) )}{\ln a_n }<\frac{\ln b_n}{\ln a_n}<1+\frac{\ln(A+
\varepsilon) )}{\ln a_n}.$$ Since
$\lim\limits_{n\to\infty}a_n=+\infty$, we  obtain  that
$$\lim_{n\to\infty}\frac{\ln b_n}{\ln a_n}=1.$$\end{proof}

\noindent{\it Proof of Theorem \ref{thm3.1}}\

By Lemma \ref{lem3.2-0}, we have for $C(d,\ell-1)<n\le C(d,\ell),\
\ell=0,1,2,\dots,$
$$a_n({\rm H}^{\Lz}(\ss))=\lz_{\ell,d}.$$
Since the function $x^{\az/d}(\ln x)^\beta\ (\az>0)$ is increasing
for sufficiently large $x$, it follows that for sufficiently large
$\ell$,
\begin{align}(C(d,\ell-1))^{\alpha/d}(\ln (C(d,\ell-1)))^{\beta} \lz_{\ell,d}
&< n^{\alpha/d}(\ln n)^{\beta}a_n({\rm H}^{\Lz}(\ss))\notag\\
&\le (C(d,\ell))^{\alpha/d}(\ln (C(d,\ell)))^{\beta}
\lz_{\ell,d}.\label{3.2}\end{align} We recall from \eqref{2.2}
that
$$C(d,\ell)= \frac{(2\ell+d)(\ell+d-1)\,!}{\ell\,!\,d\,!},
$$which yields that
$$\lim_{\ell\to\infty}\frac{C(d,\ell)}{\ell^d}
=\lim_{\ell\to\infty}\frac{C(d,\ell-1)}{\ell^d}=\frac 2{d\,!}.$$
It follows from Lemma \ref{lem3.3-0} that
$$\lim_{\ell\to\infty}\frac{\ln(C(d,\ell))}{\ln(\ell^d)}
=\lim_{\ell\to\infty}\frac{\ln(C(d,\ell-1))}{\ln(\ell^d)}=1.$$
We obtain
\begin{align}
& \quad \ \lim_{\ell\to\infty}(C(d,\ell-1))^{\alpha/d}(\ln (C(d,\ell-1)))^{\beta} \lz_{\ell,d}\notag\\
&=\lim_{\ell\to\infty}(\frac{C(d,\ell)}{\ell^d})^{\alpha/d}\Big(\frac{\ln
(C(d,\ell-1))}
{\ln (\ell^d)}\Big)^{\beta} d^\beta \ell^{\alpha} (\ln \ell)^\beta  \lz_{\ell,d}\notag\\
&=\Big(\frac2{d\,!}\Big)^{\alpha/d} d^\beta,\label{3.3}
\end{align}
and \begin{equation}\lim_{\ell\to\infty}(C(d,\ell))^{\alpha/d}(\ln
(C(d,\ell)))^{\beta}
\lz_{\ell,d}=\Big(\frac2{d\,!}\Big)^{\alpha/d}
d^\beta.\label{3.4}\end{equation} Clearly, $n\to\infty$ if and
only if $\ell\to\infty$. Equation \eqref{3.1} follows from
\eqref{3.2}, \eqref{3.3} and \eqref{3.4} immediately. This
completes the proof of Theorem \ref{thm3.1}. $\hfill\Box$

\begin{cor}
Let $\Lz=\{\lz_{\ell}\}_{\ell=0}^\infty$ be  a nonnegative non-increasing
 sequence with $\lz_{\ell} \asymp \ell^{-\az} \ln^{-\beta}
\ell$ for some $\az>0$ and $\beta\in \Bbb R$. We have
$$a_n({\rm H}^{\Lz}(\ss))\asymp n^{-\az/d}\ln^{-\beta}n.$$
\end{cor}

\

Next we show Theorem \ref{thm1.2}. In order to prove Theorem
\ref{thm1.2} we need the following lemma.

\begin{lem}(See \cite[Proposition 3.5]{DW}.) \label{lem4.1-1} Let $X$ be a linear subspace of $\Pi_n^{d+1}$ with $\dim X\ge \vz \dim\Pi_n^{d+1}$ for some $\vz\in(0, 1)$.
Then there exists a function $f\in X$ such that
$$\|f\|_{L_p}\asymp \|f\|_{L_\infty} \asymp 1, \ \ \  1\le p<\infty,$$
where $$\|f\|_{L_\infty}=\max_{\x\in\ss}|f(\x)|,\ \ {\rm and}\ \
\|f\|_{L_p}^p=\int_{\ss}|f(\x)|^pd\sz_d(\x),\ 1\le
p<\infty.$$\end{lem}

\

\noindent{\it Proof of  Theorem \ref{thm1.2}.}

Let $\x_1,\dots,\x_n$ be any $n$ distinct points on $\Bbb S^d$.
Take a positive integer $N$ such that $2n\le d_N\le Cn$ and denote
$$X_0:=\big\{g\in\Pi_N^{d+1}\ \big|\ g(\x_j)=0\ {\rm for\ all}\
j=1,\dots,n\big\},$$where $d_N={\rm dim}\,\Pi_N^{d+1}$. Thus,
$n\asymp N^d$, and $X_0$ is a linear subspace of $\Pi_N^{d+1}$
with
$${\rm dim}\,X_0\ge  d_N-n\ge d_N/2
.$$ It follows from Lemma \ref{lem4.1-1} that there exists a
function $g_0\in X_0$ such that
$$\|g_0\|_{L_p}\asymp1,\ {\rm}\ {\rm for\ all}\
1\le p \le\infty.$$ Let
$$f_0(\x)=N^{-\az}\ln^{-\beta}N (g_0(\x))^2,\ \x\in\ss.$$ Then
$$f_0\in\Pi_{2N}^{d+1}, \ \ f_0(\x_1)=\dots=f_0(\x_n)=0, \ \
f_0(\x)\ge 0\ {\rm for}\ \x\in\ss,$$and
$$\int_{\ss}f_0(\x)d\sz_d(\x)=N^{-\az}\ln^{-\beta}N\|g_0\|_{L_2}^2\asymp
N^{-\az}\ln^{-\beta}N.$$

Now we show $\|f_0\|_{{\rm H}^{\az,\beta}(\ss)}\lesssim 1$. Take
$m\in\Bbb N$ such that $2^{m-1}\le N<2^m$. Then by \eqref{2.8-0}
and the fact that $E_{2^j}(f_0)_{L_2} \le \|f_0\|_{L_2}$ for $j\le
m+1$ and $E_{2^j}(f_0)_{L_2}=0$ for $j>m+1$, we have
\begin{align*}\|f_0\|_{{\rm H}^{\az,\beta}(\ss)}&\asymp
\|f_0\|_{L_2}+\Big(\sum_{j=1}^\infty 2^{2\az
j}j^{2\beta}E_{2^j}(f)_{L_2}^2\Big)^{1/2}\\&\lesssim
\|f_0\|_{L_2}+\Big(\sum_{j=1}^{m+1} 2^{2\az
j}j^{2\beta}\|f_0\|_{L_2}^2\Big)^{1/2}\\ &\lesssim 2^{\az
m}m^{\beta}\|f_0\|_{L_2}\lesssim \|g_0\|_{L_4}^2\lesssim 1
\end{align*}

It follows from \cite[Section 5, Chapter 4]{TWW} that
 \begin{align*}
   e_n({\rm H}^{\az,\beta}(\ss))&\ge \inf_{\substack{\x_1,\dots,\x_n\in \Bbb S^d}}
   \sup_{\substack{\|f\|_{{\rm H}^{\az,\beta}(\ss)}\le 1\\ f(\x_1)=\dots=f(\x_n)=0}}\left|\int_{\Bbb S^d}f({\bf x}){\rm d}\sz_d({\bf
   x})\right|\\&\gtrsim \inf_{\substack{\x_1,\dots,\x_n\in \Bbb
   S^d}} \left|\int_{\Bbb S^d}f_0({\bf x}){\rm d}\sz_d({\bf
   x})\right|\\ &\gtrsim N^{-\az}\ln^{-\beta}N\gtrsim
   n^{-\az/d}\ln^{-\beta} n.
 \end{align*}
The proof of Theorem \ref{thm1.2} is finished. $\hfill\Box$

\section{Lower bounds for  $e_n({\rm H}^{\az,\beta}(\ss))$ with $\az=d/2,\, \beta>1/2$}

We start with a definition of convolution of two sequences and
obtain  lower bounds for $e_n({\rm H}_{\widetilde{\Lambda}})$  for
a sequence $\widetilde{\Lambda}$ which is given as a sum of
convolution squares. Then we give lower bounds for $e_n({\rm
H}_{\widetilde{\Lambda}})$ for a  nonnegative non-increasing
sequence $\widetilde{\Lambda}$. Finally, we prove Theorem
\ref{thm1.3}.

Let $w_0(t):=c_0(1-t^2)^{(d-2)/2},\ d\ge 2$ is a weight function
on $[-1,1]$, where $c_0=(\int_{-1}^1(1-t^2)^{(d-2)/2}{\rm
d}t)^{-1}$. For any $f \in L_1([-1,1],w_0)$, we  define
$$\langle f\rangle:=\int_{-1}^1 f(t)w_0(t){\rm d}t.$$
Thus, $L_2([-1,1],w_0)$ is a Hilbert space with inner product
$$\langle f,g\rangle_{2,w_0}=\int_{-1}^1 f(t)g(t)w_0(t){\rm d}t=\langle f
g\rangle,$$ and norm
$$\|f\|_{2,w_0} =\left(\int_{-1}^1|f(t)|^2w_0(t){\rm d}t\right)^{1/2}=\sqrt{\langle f^2\rangle}.$$
Let $G_\ell,\, \ell=0,1,\dots,$ be the normalized Gegenbauer
polynomials given in \eqref{2.4}. Then $G_\ell,\ \ell=0,1,\dots$,
satisfy \begin{equation}\label{4.1-0}{\rm deg}\, G_\ell=\ell
 \ \ {\rm and}\ \ \langle f G_\ell \rangle =0 \ {\rm for\ all} \ f\in \mathcal P_{\ell-1},\end{equation}
where $\mathcal P_k$ is the space of algebraic polynomials of
degree at most $k$.
 It follows
from \cite{DX} that
\begin{equation}\label{4.1-1}h_\ell^2:=\|G_{\ell}\|^{2}_{2,w_0}=\frac{\ell! (\ell
+d-1)}{(d)_{\ell}(2\ell+d+1)} \, \, \asymp \ell^{-d+1}, \ \
(d)_\ell=d(d+1)\cdots(d+\ell-1).\end{equation}
 We have
\begin{align*}
  \frac{h_\ell^2}{h_{\ell+1}^2}&=\frac{\ell !(\ell+d-1)(d)_{\ell+1}(2\ell+d+3)}{(\ell+1) !(\ell+d)(d)_\ell(2\ell+d+1)}
  =\frac{(\ell+d-1)(2\ell+d+3)}{(\ell+1) (2\ell+d+1)}>1,
\end{align*}which means that $h_\ell^2$ is decreasing.
Obviously,  $\{G_\ell/h_\ell\}_{\ell=0}^\infty$ forms an
orthonormal basis of $L_2([-1,1],w_0)$. It follows that
\begin{align}
G_{\ell}G_{s}=\sum_{k=0}^{\infty}\big\langle G_{\ell}G_{s},
\frac{G_{k}}{h_k}\big\rangle_{2,w_0}\frac{G_{k}}{h_k}=\sum_{k=0}^{\infty}\frac{\langle
G_{\ell}G_{s} G_{k}\rangle}{h^{2}_k} G_{k}
=\sum_{k=0}^{\infty}C_k^{\ell,s}G_k, \label{4.1}
\end{align}
where $C_k^{\ell,s}:={\langle G_\ell G_sG_k\rangle}/h_k^2$.
Moreover, $C_k^{\ell,s}$ have the following properties.
\begin{prop}\label{prop2}For any $\ell,s,k=0,1,\dots$, we have
\begin{enumerate}
  \item $C_k^{\ell,s}=C_k^{s,\ell}$;
  \item $C_k^{\ell,s}\ge0$ if $|\ell-s|\le k\le\ell+s$;
  \item $C_k^{\ell,s}=0$ if $k>\ell+s$ and $k< |\ell-s|$;
  \item $\sum_{k=|\ell-s|}^{\ell+s}C_k^{\ell,s}=1$;
  \item $\sum_{k=|\ell-s|}^{\ell+s}\langle G_\ell G_sG_k\rangle^2/h_k^2=\langle G_\ell^2 G_s^2\rangle$.
\end{enumerate}
\end{prop}

\begin{proof} Property (1) follows from the definition of
$C_k^{\ell,s}$. Property (2) follows from \cite{Ga,Hs}.  Property
(3) is due to \eqref{4.1-0} and the fact that $G_\ell
G_s\in\mathcal P_{\ell+s}$. Property (4) comes from \eqref{4.1},
Property (3), and $G_\ell(1)=1$. Property (5) follows from
\eqref{4.1}, Property (3),  and the Parseval equality.
\end{proof}

Let $\widetilde\mu$ be a nonnegative sequence in $\ell_1(\Bbb
N_0)$. We set
$$f_{\widetilde\mu}(t)=\sum_{j=0}^\infty  \widetilde\mu_j G_j(t),
\ \ \ {\rm and}\ \ \
K_{\widetilde\mu}(\x,\y)=f_{\widetilde\mu}(\x\cdot\y), \
\x,\y\in\ss.$$ Then ${\rm H}_{\widetilde\mu}\equiv {\rm
H}_{\widetilde\mu}(\ss)$ is a RKHS with the reproducing kernel
$K_{\widetilde\mu}$, and for all $\x\in\ss$,
$$K_{\widetilde\mu}(\x,\x)=\|\widetilde\mu\|_{\ell_1}, \ \ \ h_{\widetilde\mu}(\x)=\widetilde
\mu_0 $$where $h_{\widetilde\mu}\in {\rm H}_{\widetilde\mu}$ is
the representer of the integration problem INT for ${\rm
H}_{\widetilde\mu}$.

For two nonnegative sequences $\widetilde\mu, \widetilde\nu$ in
$\ell_{1}(\Bbb N_{0})$, we define the formal convolution
$\widetilde\mu * \widetilde\nu=\{(\widetilde\mu *
\widetilde\nu)_k\}_{k=0}^\infty$ by
$$ (\widetilde \mu \ast  \widetilde \nu)_{k}
=\sum_{\ell=0}^{\infty}\sum_{s=0}^{\infty} C_k^{\ell,s}\widetilde
\mu_{\ell}\widetilde \nu_{s}.$$ It follows from \eqref{4.1} that
$$f_{\widetilde \mu \ast  \widetilde \nu}(t)=f_{\widetilde
\mu}(t)f_{\widetilde \nu}(t),\ \ {\rm and}\  \ K_{\widetilde \mu
\ast \widetilde \nu}(\x,\y)=K_{\widetilde \mu}(\x,\y)K_{\widetilde
\nu}(\x,\y)$$ for all $t\in[-1,1]$ and $\x,\y\in\ss.$
Specifically, we have
$$ K_{\widetilde \mu
\ast \widetilde \mu}(\x,\y)=K_{\widetilde \mu}(\x,\y)^2.$$

\begin{thm}\label{thm4.2}Let $\widetilde\mu, \widetilde\nu$ be  two nonnegative  sequences
 in $\ell_{1}(\Bbb N_{0})$, and let $\widetilde{\Lambda}=\{\widetilde\lz_k\}_{k=0}^\infty$ satisfy $\widetilde{\Lambda}=\widetilde\mu*
 \widetilde\mu+\widetilde\nu*\widetilde\nu$.
  Then we have
$$e_{n}({\rm H}_{\widetilde\Lambda})^{2}\geq  \widetilde\lz_0 \left(1- \frac{n\widetilde\lz_0}{\sum_{\ell=0}^{\infty}\widetilde{\lz}_{\ell}}\right).$$
\end{thm}
\begin{proof}
For $n=0$, we have the initial error satisfies
 $$e_{0}({\rm H}_{\widetilde\Lambda})^{2}=\|h_{\widetilde\Lambda}\|_{{\rm H}_{\widetilde\Lambda}}^{2}={\widetilde\lz_0},
 $$where $h_{\widetilde\Lambda}(\x)=\widetilde\lz_0$ is the representer of the integration
problem INT for ${\rm H}_{\widetilde\Lz}$.

For $n=1,2,\dots$, it follows from Lemma 2.3 that for all
$\x_{1},...,\x_{n}\in \ss$, we know the matrixes
$(K_{\widetilde\mu}(\x_{j},\x_{k}))^{n}_{j,k=1}$ and
$(K_{\widetilde\nu}(\x_{j},\x_{k}))^{n}_{j,k=1}$  are positive
semi-definite, and by Lemma 2.5, the matrix with the entries
$$K_{\widetilde\mu}(\x_{j},\x_{k})^2- \frac{K_{\widetilde\mu}(\x_{j},\x_{j})K_{\widetilde\mu}(\x_{k},\x_{k})}{n}
=K_{\widetilde\mu
*\widetilde\mu}(\x_{j},\x_{k})-\frac{\|\widetilde\mu\|_{\ell_{1}}^2}{n}$$is
positive semi-definite. Similarly, the matrix with the entries
$$K_{\widetilde\nu
*\widetilde\nu}(\x_{j},\x_{k})-\frac{\|\widetilde\nu\|_{\ell_{1}}^2}{n}$$
is  positive semi-definite. It follows that the matrix with the
entries $$K_{\widetilde\mu
*\widetilde\mu}(\x_{j},\x_{k})+K_{\widetilde\nu
*\widetilde\nu}(\x_{j},\x_{k})-\frac{\|\widetilde\mu\|_{\ell_{1}}^2}{n}-\frac{\|\widetilde\nu\|_{\ell_{1}}^2}{n}$$
is also positive semi-definite.  We note that
$$K_{\widetilde\Lz}(\x,\y)= K_{\widetilde\mu
*\widetilde\mu}(\x,\y)+ K_{\widetilde\nu *\widetilde\nu}(\x,\y)\ \
{\rm and}\ \
\|\widetilde\Lz\|_{\ell_1}=\|\widetilde\mu\|_{\ell_1}^2+\|\widetilde\nu\|_{\ell_1}^2.$$Hence,
the matrix with the entries
$$K_{\widetilde\Lz
}(\x_{j},\x_{k})-\frac{\|\widetilde\Lz\|_{\ell_{1}}}{n}=K_{\widetilde\Lz
}(\x_{j},\x_{k})-\frac{\|\widetilde\Lz\|_{\ell_{1}}}{n\widetilde
\lz_0^2}h_{\widetilde\Lz}(\x_j)h_{\widetilde\Lz}(\x_k)$$ is also
positive semi-definite.
 Thus applying  Lemma 2.4 and the fact that   $\|h_{\widetilde\Lz}\|_{{\rm H}_{\widetilde\Lz}}^{2}=\widetilde\lz_0$, we have
$$
 e_{n}({\rm H}_{\widetilde\Lambda})^{2} \geq \|h_{\widetilde\Lz}\|_{{\rm H}_{\widetilde\Lambda}}^{2} -\frac{n\widetilde\lz^2_{0}}{\|\widetilde\Lz\|_{\ell_{1}}}
 =\widetilde \lz_0 \left(1- \frac{n\widetilde\lz_0}{\sum_{\ell=0}^{\infty}\widetilde{\lz}_{\ell}}\right).
$$
This completes the proof of Theorem \ref{thm4.2}.
\end{proof}

\begin{thm}\label{thm4.8}
 Let $\widetilde{\Lambda}=\{\widetilde{\lambda}_\ell\}_{\ell=0}^\infty$
 be a nonnegative non-increasing sequence in $\ell_1(\Bbb N_0)$.
Then
$$e_{n}({\rm H}_{\widetilde\Lambda})^{2}\geq  \widetilde\lz_0 \left(1- \frac{2cn\widetilde\lz_0}{\sum_{\ell=1}^{\infty}
\widetilde{\lambda}_{2\ell}+2c\widetilde\lz_0}\right),$$ where
$c>1$ is an absolute constant such that
$h_{\lfloor\ell/2\rfloor}^2\le ch_\ell^2$  for all $\ell\in\Bbb
N_0$, $\lfloor a\rfloor$ denotes the largest integer not exceeding
$a$, and $h_\ell^2$ is given in \eqref{4.1-1}.
\end{thm}

\begin{proof} We set
$$
\widetilde{\mu}=\{\widetilde{\mu}_\ell\}_{\ell=0}^\infty,\ \ \
\widetilde{\mu}_\ell=(2cA )^{-1/2}\widetilde{\lz}_{2\ell},\
A=\sum_{\ell=0}^{\infty}{\widetilde\lz}_{2\ell},$$ and
$$\dz_0=\{\delta_{0\ell}\}_{\ell=0}^\infty,\ \ \ \dz_{00}=1,\ {\rm and}\ \dz_{0\ell}=0\ {\rm for}\ \ell\ge1,$$where $c>1$ is
an absolute constant such that $h_{\lfloor\ell/2\rfloor}^2\le ch_\ell^2$. Such
constant $c$ exists since $h_\ell^2\asymp (\ell+1)^{-d+1}$. We
note that $$\dz_0*\dz_0=\dz_0.$$

We have for $k\ge1$,
\begin{align*}
(\widetilde\mu \ast  \widetilde\mu)_{k}&=
\sum_{\ell=0}^{\infty}\sum_{s=0}^{\infty}
C_k^{\ell,s}\widetilde\mu_{\ell}\widetilde\mu_{s}\notag =(2cA
)^{-1}\sum_{\ell=0}^{\infty}\sum_{s=0}^{\infty} C_k^{\ell,s}
{\widetilde\lz}_{2\ell}{\widetilde\lz}_{2s}\notag
\\
&\leq\frac{(2cA
)^{-1}}{h_k^2}\left\{\sum_{\ell=0}^{\infty}\sum_{s\ge k/2}
C_s^{\ell,k}{\widetilde\lz}_{2\ell}\widetilde\lz_{2s}h_s^2+\sum_{s=0}^{\infty}\sum_{\ell\ge
k/2}
C_\ell^{s,k}{\widetilde\lz}_{2s}\widetilde\lz_{2\ell}h_\ell^2\right\}\notag
\\&\le\frac{(2cA )^{-1}}{h_k^2}\left\{\sum_{\ell=0}^{\infty}\left(\sum_{s\ge k/2}
C_s^{\ell,k}\right){\widetilde\lz}_{2\ell}+\sum_{s=0}^{\infty}\left(\sum_{\ell\ge
k/2}
C_\ell^{s,k}\right){\widetilde\lz}_{2s}\right\}{\widetilde\lz}_{k}h_{\lfloor
k/2\rfloor}^2\notag
\\&\le (2cA )^{-1}2\frac{h_{\lfloor k/2\rfloor}^2}{h_k^2}\left(\sum_{\ell=0}^{\infty}{\widetilde\lz}_{2\ell}\right)\widetilde\lz_{k}\notag\le\widetilde\lambda_{k},
\end{align*}
where in the first inequality we used the fact that $
C_k^{\ell,s}\neq 0$ implies $s\ge k/2$ or ${\ell\ge k/2} $ and
that $C_k^{\ell,s}h_k^2=\langle G_k G_\ell G_s\rangle$,    in the
second inequality we  used    $\widetilde{\Lambda}$ is a  nonnegative
non-increasing sequence, in the last second inequality we used
Proposition \ref{prop2} (4). And for $k=0$ we have
\begin{align*}
  (\widetilde\mu \ast  \widetilde\mu)_{0}&=\sum_{\ell=0}^{\infty} \frac{h_\ell^2}{h^{2}_0}\tilde \mu^{2}_{\ell}
  =(2cA )^{-1}\sum_{\ell=0}^{\infty} \frac{h_\ell^2}{h^{2}_0} \widetilde{\lz}_{2\ell}^2\\ &\le (2cA )^{-1}
  \left(\sum_{\ell=0}^{\infty}{\widetilde\lz}_{2\ell}\right)\widetilde\lz_{0}=
  (2c)^{-1}\widetilde\lz_{0}.
\end{align*}

 Now we put $\widetilde \nu =\widetilde \mu \ast  \widetilde \mu+ t\,\delta_{0}*\dz_0$
 and choose $t\geq (1-(2c)^{-1})\widetilde\lz_0$ such that  $\widetilde
 \nu_{0}=\widetilde\lz_{0}$.
 Then $\widetilde \nu$ is a sum of two convolution squares, and $\widetilde\lz_\ell\ge \widetilde\mu_\ell$ for all $\ell\in \Bbb N_0$. It follows from
 Theorem 4.2
  that
\begin{equation}\label{4.1-3}e_{n}({\rm H}_{\widetilde\Lz})^{2}\geq e_{n}({\rm
H}_{\widetilde \nu})^{2} \geq \widetilde \nu_0 \left(1-
\frac{n\widetilde\nu_0}{\|\widetilde
\nu\|_{\ell_{1}}}\right)=\widetilde \lz_0 \left(1-
\frac{n\widetilde\lz_0}{\|\widetilde
\nu\|_{\ell_{1}}}\right),\end{equation} where
\begin{align*}
\|\widetilde \nu\|_{\ell_{1}} &= \|\widetilde \mu\|_{\ell_{1}}^2+
t\|\delta_{0}\|_{\ell_1}^2 \notag\geq (2cA
)^{-1}\(\sum_{\ell=0}^{\infty}{\widetilde\lz}_{2\ell}\)^2+(1-(2c)^{-1})\widetilde\lz_0\\&
=(2c)^{-1}\sum_{\ell=0}^{\infty}{\widetilde\lz}_{2\ell}+(1-(2c)^{-1})\widetilde\lz_0=(2c)^{-1}\sum_{\ell=1}^{\infty}{\widetilde\lz}_{2\ell}+\widetilde\lz_0.
\end{align*}By \eqref{4.1-3} we obtain
$$e_{n}({\rm H}_{\widetilde\Lz})^{2}\geq \widetilde \lz_0 \left(1-
\frac{2cn\widetilde\lz_0}{\sum_{\ell=1}^{\infty}{\widetilde\lz}_{2\ell}+2c\widetilde\lz_0}\right).
$$ This completes the proof of Theorem \ref{thm4.8}.
\end{proof}

\begin{thm}\label{thm4.9}
Let
$\widetilde{\Lambda}=\{\widetilde{\lambda}_\ell\}_{\ell=0}^\infty$
be a nonnegative non-increasing sequence in $\ell_1(\Bbb N_0)$.
Then
$$e_{n}({\rm H}_{\widetilde\Lambda})^2\geq \min
\left\{\frac12\widetilde\lz_0, \frac{1}{ 4cn}\sum_{\ell>
2cn}\widetilde \lambda_{2\ell}\right\},$$ where $c>1$ is an
absolute constant such that
 $h_{\lfloor\ell/2\rfloor}^2\le ch_\ell^2$  for all $\ell\in\Bbb N_0$.
\end{thm}
\begin{proof}
Let  $\tau=\{\tau_\ell \}_{\ell=0}^\infty$ defined by
$$\tau_{0}=\min \left\{\widetilde\lz_0, \,\max\left\{ \frac{1}{2}\widetilde
\lambda_{\lfloor 4cn\rfloor}, \frac{1}{2cn}\sum_{\ell>
2cn}\widetilde \lambda_{2\ell}\right\}\right\}$$ and
$$\tau_{\ell} :=\left\{\begin{aligned}& \widetilde\lambda_{\ell},
   &&{\rm for}\ \ \ell>  4cn,\\&
 \widetilde \lambda_{\lfloor 4cn\rfloor},     &&{\rm for}\
\ 0<\ell\leq  4cn.
\end{aligned}\right.
$$
Then $\tau$ is a nonnegative non-increasing sequence in
$\ell_1(\Bbb N_0)$. Moreover, we have $\widetilde\lz_\ell
>\tau_\ell$ for all $\ell\in\Bbb N_0$, and
\begin{align*}
\sum^{\infty}_{\ell=1}\tau_{2\ell}+2c\tau_{0}&\geq
\sum_{2\ell> 4cn}\widetilde\lambda_{2\ell}+\sum_{0<2\ell\leq 4cn}\widetilde\lambda_{\lfloor 4cn\rfloor}+2c\tau_{0}\notag \\
&\geq \sum_{\ell> 2cn}\widetilde\lambda_{2\ell}+
cn\widetilde\lambda_{\lfloor 4cn\rfloor}\geq4cn\tau_{0}.
\end{align*}
Applying Theorem \ref{thm4.8}, we have

\begin{align*}
e_{n}({\rm H}_{\widetilde\Lz})^{2}\geq e_{n}({\rm H}_{\tau})^{2}
&\geq \tau_0 \left(1-
\frac{2cn\tau_0}{\sum^{\infty}_{\ell=1}\tau_{2\ell}
+2c\tau_{0}}\right)\notag \geq \frac{1}{2}\tau_0\notag
\\&\ge\min \left\{\frac12\widetilde\lz_0, \frac{1}{ 4cn}\sum_{\ell>2cn}\widetilde \lambda_{2\ell}\right\}.
\end{align*}\
This completes the proof of Theorem \ref{thm4.9}.
\end{proof}

\noindent{\it Proof of Theorem \ref{thm1.3}}\

We consider the space ${\rm H}^{\az,\beta}(\ss)={\rm
H}^{\Lz^{(\az,\beta)}}(\ss)={\rm H}_{\widetilde\Lz^{(\az,\beta)}}$
with $\az=d/2,\,\beta>1/2$, where
$$\Lz^{(\az,\beta)}=\{\lz_\ell^{(\az,\beta)}\}_{\ell=0}^\infty,\ \
\
\widetilde\Lz^{(\az,\beta)}=\{\widetilde\lz_\ell^{(\az,\beta)}\}_{\ell=0}^\infty,$$
$$\lz_\ell^{(\az,\beta)}=(1+\ell(\ell+d-1))^{-d/4}(\ln(3+\ell(\ell+d-1)))^{-\beta},
$$and
$$\widetilde\lz_\ell^{(\az,\beta)}=(\lz_\ell^{(\az,\beta)})^2N(d,\ell)\asymp
\ell^{-1}\ln^{-2\beta}\ell.
$$
It follows from Theorem \ref{thm4.9}  that
 $$e_n({\rm H}^{\az,\beta}(\ss))^{2}
\gtrsim \frac{1}{4cn}\sum_{\ell>2cn}\widetilde
\lambda^{(\az,\beta)}_{2\ell} \asymp n^{-1}\sum_{\ell>
2cn}\ell^{-1}\ln^{-2\beta}\ell\asymp n^{-1}\ln^{-2\beta+1}\ell.$$
This completes the proof of Theorem \ref{thm1.3}. $\hfill\Box$

\
\section{Proof of Theorems \ref{thm1.5}}\

In order to prove Theorem \ref{thm1.5}, we shall use the following lemma.

\begin{lem}\label{lem5.1} (See \cite[Theorem 2.1]{D2}.)
Let  $\Ga$ be   a finite subset of  $\,\sph$, and  let  $\{
\mu_{\boldsymbol \omega}:\ \ \boldsymbol\omega\in \Ga\}$ be   a
set of positive numbers satisfying
\begin{equation*} \sum_{ \boldsymbol\omega\in\Ga} \mu_{\boldsymbol\omega}
|f(\boldsymbol\omega)|^{p_0} \leq C_1 \int_{\sph} |f(\x)|^{p_0}\,
d\sz_d(\x),\ \ \forall f\in \Pi_{N}^{d+1},
\end{equation*} for some $0<p_0<\infty$ and some positive integer $N$. If
 $0< q<\infty$, $M\ge
N$ and $f\in \Pi_M^{d+1}$, then
\begin{equation*} \sum_{ \boldsymbol\omega\in \Ga}
\mu_{\boldsymbol\omega} |f(\boldsymbol\omega)|^q \leq C C_1 \( \f
MN\)^{d} \int_{\sph} |f(\y)|^q \, d\sz_d(\y),\end{equation*} where
$C>0$ depends only on $d$ and $q$.
\end{lem}

\noindent{\it Proof of Theorem \ref{thm1.5}}\

By \eqref{2.33-4} it suffices to give the upper bounds for
$\|f-L_N^{\Bbb S}(f)\|_{L_2}$. For $f\in {\rm H}^{\az,\beta}(\ss)$
with $\az>d/2$ or $\az=d/2,\,\beta>1/2$, we define
$$\sigma_1(f)=S_{2}(f), \ \ \sigma_j(f)=S_{2^{j}}(f)-S_{2^{j-1}}(f),\ \
{\rm for}\ j\ge2.$$Then $\sigma_j(f)\in\Pi_{2^{j}}^{d+1}$ and
$$f=\sum_{j=1}^\infty\sigma_j(f)$$ converges  in the uniform norm.

For a positive integer $N$, let  $m$ be an integer such that
$2^m\le N<2^{m+1}$. Observe that for any $f\in{\rm
H}^{\az,\beta}(\ss)$
\begin{align}\label{6.1}
  \|f-L_N^{\Bbb S}(f)\|_{L_{2}}&\le
  \|f-S_{2^{m}}(f)\|_{L_{2}}+\|L_{N}^{\Bbb S}(f)-S_{2^{m}}(f)\|_{L_2}\notag
  \\&=:{\rm I}+{\rm II}.
\end{align}

To estimate  ${\rm I}=\|f-S_{2^{m}}(f)\|_{L_2}$, we can use
\eqref{2.5-1} to obtain
\begin{align}\label{6.2}
  {\rm I}\ &\lesssim   \left(2^{m}\right)^{-\alpha}\ln^{-\beta}(1+2^{m})\,\|f\|_{{\rm H}^{\az,\beta}(\ss)}\notag
  \\&\lesssim N^{-\alpha}\ln^{-\beta}N\,\|f\|_{{\rm H}^{\az,\beta}(\ss)}.
\end{align}

To estimate $\text{II} = \|L_N^{\Bbb S}(f)-S_{2^{m}}(f)\|_{L_2}$,
we can use \eqref{2.33-1} and \eqref{2.33-3} with
$S_{2^m}(f)\in\Pi_N^{d+1}$ to obtain that
\begin{align}
  {\rm II}\ &\le  A^{-1/2}\|L_N^{\Bbb S}(f)-S_{2^{m}}(f)\|_{(N)}
  \notag\\&\le A^{-1/2}\{\|L_N^{\Bbb S}(f)-f\|_{(N)}+\|f-S_{2^{m}}(f)\|_{(N)}\}
  \notag\\&\le 2 A^{-1/2}\|f-S_{2^{m}}(f)\|_{(N)},\label{6.3}
\end{align}where we use the fact that $L_N^{\Bbb S}(f)-S_{2^{m}}(f)\in \Pi_N^{d+1}$.
Since for all $f\in {\rm H}^{\az,\beta}(\ss)$, the series
$\sum_{j=m+1}^\infty\sigma_j(f)$ converges to $f-S_{2^{m}}(f)$ in
the uniform norm, we have
 \begin{equation}\label{6.4}
 \|f-S_{2^{m}}(f)\|_{(N)}\le \sum\limits_{j=m+1}^\infty\|
  \sigma_j(f)\|_{(N)}.
 \end{equation}

For $j\ge m+1$, using Lemma \ref{lem5.1} with $p_0=q=2$ we get
\begin{align}\label{6.5}
\| \sigma_j(f) \|_{(N)}^2 &= \sum_{k=1}^{l_N} \tau_{N,k} |
\sigma_j(f)(\mathbf{x}_{N,k})|^2 \lesssim B {2^{jd}}N^{-d} \|
\sigma_j(f) \|_{L_2}^2.
\end{align}
Applying \eqref{6.3}, \eqref{6.4},  \eqref{6.5}, the H\"older
inequality, and \eqref{2.8} we have
\begin{align}\label{6.6}
  {\rm II}&\le A^{-1/2}\sum_{j=m+1}^\infty \|\sigma_j(f)\|_{(N)}\notag\\&\lesssim \kappa^{1/2} N^{-d/2}\sum_{j=m+1}^\infty{2^{jd/2}} \|
\sigma_j(f) \|_{L_2}\notag
  \\&\lesssim \kappa^{1/2} N^{-d/2}\left(\sum_{j=m+1}^\infty2^{2j\az}j^{2\beta}\|
\sigma_j(f) \|_{L_2}^2\right)^{1/2}\
\left(\sum_{j=m+1}^\infty2^{-2j\az+jd}j^{-2\beta}\right)^{1/2}
\notag
   \\&\lesssim \kappa^{1/2} N^{-d/2}\left(\sum_{j=m+1}^\infty2^{-2j\az+jd}j^{-2\beta}\right)^{1/2}\ \|f\|_{{\rm H}^{\az,\beta}(\ss)}\notag
  \\&\lesssim \left\{\begin{aligned}&\kappa^{1/2} N^{-d/2}m^{-\beta+1/2} \ \|f\|_{{\rm H}^{\az,\beta}(\ss)}, &&\alpha=d/2,\,\beta>1,
  \\&\kappa^{1/2} N^{-d/2}m^{-\beta}2^{-(\alpha-d/2) m}\ \|f\|_{{\rm H}^{\az,\beta}(\ss)}, &&\alpha>d/2,\ \beta\in \Bbb R,
  \end{aligned}\right.\notag
  \\& \lesssim \kappa^{1/2} C_{\az,\beta}(N)  \|f\|_{{\rm H}^{\az,\beta}(\ss)} \end{align}
 Therefore, it follows from \eqref{6.1}, \eqref{6.2}, and
\eqref{6.6} that
$$\|f-L_N^{\Bbb S}(f)\|_{L_{2}}\lesssim (1+\kappa^{1/2}) C_{\az,\beta}(N)  \|f\|_{{\rm H}^{\az,\beta}(\ss)},$$
where
$$C_{\az,\beta}(N):=\bigg\{\begin{array}{ll} N^{-\az}\ln^{-\beta}N,\ \ & \az>d/2, \\
N^{-\az}\ln^{-\beta+1/2}N, \ &\az=d/2, \beta>1/2.\end{array}$$

This completes the proof of Theorem \ref{thm1.5}. $\hfill\Box$

\

 \noindent{\bf Acknowledgment}   The authors  were
supported by the National Natural Science Foundation of China
(Project no. 12371098).

\end{document}